\documentclass[a4paper,12pt,reqno]{amsart}
\usepackage{a4wide}
\usepackage{amsmath}
\usepackage{amssymb}
\usepackage{amsthm}
\usepackage{latexsym}
\usepackage{graphicx}
\usepackage[english]{babel}

\newtheorem{obs} [subsection]{Remark}
\newtheorem{exm} [subsection]{Example}

\newtheorem{prop}[subsection]{Proposition}
\newtheorem{conj}[subsection]{Conjecture}
\newtheorem{teor}[subsection]{Theorem}
\newtheorem*{teor*}{Theorem}
\newtheorem{lema}[subsection]{Lemma}
\newtheorem{cor} [subsection]{Corollary}

\newcommand{\Zng}{$\mathbb Z^n$-graded $S$-module}

\def\supp{\operatorname{supp}}

\def\pd{\operatorname{pd}}

\def\depth{\operatorname{depth}}
\def\sdepth{\operatorname{sdepth}}
\def\qdepth{\operatorname{hdepth}}
\def\hdepth{\operatorname{hdepth}}

\usepackage{mathtools}

\begin{document}
\selectlanguage{english}
\frenchspacing
\numberwithin{equation}{section}

\title[Remarks on the sdepth and hdepth of monomial ideals with linear quot.]
{Remarks on the Stanley depth and Hilbert depth of monomial ideals with linear quotients}

\author[Andreea I.\ Bordianu, Mircea Cimpoea\c s 
       ]
  {Andreea I.\ Bordianu$^1$ and Mircea Cimpoea\c s$^2$
	}
\date{}

\keywords{Stanley depth; Hilbert depth; depth; monomial ideals; linear quotients}

\subjclass[2020]{05E40; 06A17; 13A15; 13C15; 13P10}

\footnotetext[1]{ \emph{Andreea I.\ Bordianu}, University Politehnica of Bucharest, Faculty of
Applied Sciences, 
Bucharest, 060042, E-mail: andreea.bordianu@stud.fsa.upb.ro}
\footnotetext[2]{ \emph{Mircea Cimpoea\c s}, University Politehnica of Bucharest, Faculty of
Applied Sciences, 
Bucharest, 060042, Romania and Simion Stoilow Institute of Mathematics, Research unit 5, P.O.Box 1-764,
Bucharest 014700, Romania, E-mail: mircea.cimpoeas@upb.ro,\;mircea.cimpoeas@imar.ro}

\begin{abstract}
We prove that if $I$ is a monomial ideal with linear quotients in a ring of polynomials $S$ in $n$ indeterminates
and $\depth(S/I)=n-2$, then $\sdepth(S/I)=n-2$ and, if $I$ is squarefree, $\qdepth(S/I)=n-2$. 

Also, we prove that $\sdepth(S/I)\geq \depth(S/I)$ for a monomial
ideal $I$ with linear quotients which satisfies certain technical conditions.
\end{abstract}

\maketitle

\section{Introduction}

Let $K$ be a field and let $S=K[x_1,x_2,\ldots,x_n]$ be the ring of polynomials in $n$ variables.
Let $M$ be a \Zng. A \emph{Stanley decomposition} of $M$ is a direct sum 
$$\mathcal D: M = \bigoplus_{i=1}^rm_i K[Z_i],$$ as $K$-vector spaces, 
where $m_i\in M$ are homogeneous, $Z_i\subset\{x_1,\ldots,x_n\}$ such that $m_i K[Z_i]$ is a free $K[Z_i]$-module; $m_iK[Z_i]$ is
called a \emph{Stanley subspace} of $M$. We define $\sdepth(\mathcal D)=\min_{i=1}^r |Z_i|$ and 
$$\sdepth(M)=\max\{\sdepth(\mathcal D)\;:\;\mathcal D\text{ is a Stanley decomposition of }M\}.$$
The number $\sdepth(M)$ is called the \emph{Stanley depth} of $M$. Herzog Vl\u adoiu and Zheng \cite{hvz} proved
that this invariant can be computed in a finite number of steps, when $M=I/J$, where $J\subset I\subset S$ 
are monomial ideals.

We say that the multigraded module $M$ satisfies the \emph{Stanley inequality} if $$\sdepth(M)\geq \depth(M).$$
Stanley conjectured in \cite{stanley} that $\sdepth(M)\geq \depth(M)$,
for any \Zng $\;M$. In fact, in this form, the conjecture was stated by Apel in \cite{apel}.
The Stanley conjecture was disproved by Duval et. al \cite{duval}, in the case $M=I/J$, where $(0)\neq J\subset I\subset S$ 
are monomial ideals, but it remains open in the case $M=I$, a monomial ideal.

A monomial ideal $I\subset S$ has \emph{linear quotients}, if there exists $u_1 \leqslant u_2 \leqslant \cdots  \leqslant u_m$, an ordering  
on the minimal set of generators $G(I)$, such that, for any $2\leq j\leq m$, 
the ideal $(u_1,\ldots,u_{j-1}):u_j$ is generated by variables. 

Given a monomial ideal with linear quotients $I\subset S$, Soleyman Jahan 
\cite{jah} noted that $I$ satisfies the Stanley inequality, i.e. $$\sdepth(I)\geq \depth(I).$$ 
However, a similar result for $S/I$, if true, is more difficult to prove, only some particular cases being known. 
For instance, Seyed Fakhari \cite{fakh} proved 
the inequality $$\sdepth(S/I)\geq \depth(S/I)$$ for weakly polymatroidal ideals $I\subset S$, 
which are monomial ideals with linear quotients.

In Theorem \ref{teo1}, we prove
that if $I\subset S$ is a monomial ideal with linear quotients with $\depth(S/I)=n-2$, then $\sdepth(S/I)=n-2$. In Theorem \ref{teo2}, we
prove that if $I\subset S$ is a monomial ideal with linear quotients which has a Stanley decomposition which satisfies certain conditions,
then $\sdepth(S/I)\geq \depth(S/I)$. 
Also, we conjecture that for any monomial ideal $I\subset S$ with linear quotients, there is a variable $x_i$
such that $\depth(S/(I,x_i))\geq \depth(S/I)$ and $\sdepth(S/(I,x_i))\leq \sdepth(S/I)$. 
In Theorem \ref{teo3} we prove that if this conjecture is true, then $\sdepth(S/I)\geq \depth(S/I)$,
for any monomial ideal $I\subset S$ with linear quotients. 

Given a finitely graded $S$-module $M$, its Hilbert depth is
$$\hdepth(M)=\max\left\{r\;:\;  \substack{\text{ There exists a f.g. graded }S\text{-module }N \\ \text{ with }H_M(t)=H_N(t)\text{ and }\depth(N)=r} 
 \right\}.$$
It is well known that $\hdepth(M)\geq \sdepth(M)$. See \cite{bruns} for further details.

Let $0\subset I\subsetneq J\subset S$ be two squarefree monomial ideals. 
For any $0\leq j\leq n$, we let $\alpha_j(J/I)$ to be the number of squarefree monomials $u\in S$ of degree $j$ such that
$u\in J\setminus I$. (In particular, $\alpha_j(I)$ is the number of squarefree monomials of degree $j$ which belong to $I$
and $\alpha_j(S/I)=\binom{n}{j}-\alpha_j(I)$ is the number of squarefree monomials of degree $j$ which do not belong to $I$.)

Also, for $0\leq k\leq d\leq n$, we let
\begin{equation}\label{betak}
\beta_k^d(J/I)=\sum_{j=0}^k (-1)^{k-j} \binom{d-j}{k-j}\alpha_j(J/I).
\end{equation}
(In particular, $\beta_k^d(S/I)=\sum_{j=0}^k (-1)^{k-j} \binom{d-j}{k-j}\alpha_j(S/I)$ and
$\beta_k^d(I)=\sum_{j=0}^k (-1)^{k-j} \binom{d-j}{k-j}\alpha_j(I)$.)

From \eqref{betak}, using an inversion formula, it follows that
\begin{equation}\label{alfak}
\alpha_k(J/I)=\sum_{j=0}^k \binom{d-j}{k-j}\beta_j^d(J/I)\text{ for all }0\leq k\leq d\leq n.
\end{equation}
With the above notations, we proved in \cite[Theorem 2.4]{lucrare1} that
$$\hdepth(J/I)=\max\{d\;:\;\beta_k^d(J/I)\geq 0\text{ for all }0\leq k\leq d\}.$$
If $I\subset S$ is a proper squarefree monomial ideal, we claim that
\begin{equation}\label{alfakk}
\qdepth(S/I)\leq \max\{k\;:\;\alpha_k(S/I)>0\}.
\end{equation} 
Note that $\alpha_n(S/I)=0$, since $x_1\cdots x_n\in I$, and thus $m:=\max\{k\;:\;\alpha_k(S/I)>0\}<n$. 
From \eqref{alfakk} it follows that 
$$\alpha_{m+1}(S/I)=\sum_{j=0}^{m+1}\beta^{m+1}_j(S/I).$$
Since $I\neq S$ it follows that $1\notin I$ and thus $\beta^{m+1}_0(S/I)=\alpha_0(S/I)=1$.
The above identity implies that there exists some $1\leq k\leq m+1$ with $\beta_k^{m+1}(S/I)<0$ and
therefore $\qdepth(S/I)\leq m$, as required.

Also, we will make use of the well known fact that
\begin{equation}\label{qdep}
\qdepth(J/I)\geq \sdepth(J/I).
\end{equation}
In Section \ref{s3} of our paper we study the Hilbert depth of $S/I$, where $I$ is a squarefree monomial ideal with
linear quotients. In Proposition \ref{p22} we compute the numbers $\beta_k^d(I)$'s and $\beta_k^d(S/I)$'s. 
In Corollary \ref{cor23}, we express these numbers in combinatorial terms, thus showing the difficulty 
in finding explicit formulas for $\qdepth(I)$ and $\qdepth(S/I)$. 

The main result of this section is Theorem \ref{teo4}, in which we show that if $I$ is a squarefree monomial ideal with
linear quotients with $\depth(S/I)=n-2$ then $$\qdepth(S/I)=\sdepth(S/I)=\depth(S/I)=n-2.$$

\section{Main results}\label{s2}

Let $I\subset S$ be a monomial ideal and let $G(I)$ be the set of minimal monomial generators
of $I$. We recall that $I$ has \emph{linear quotients}, if there exists a linear order
 $u_1 \leqslant u_2 \leqslant \cdots  \leqslant u_m$  on $G(I)$, such that 
for every $2 \leq j \leq m$, the ideal $(u_1, \ldots , u_{j-1}) : u_j$ is generated by a subset of $n_j$ variables.

We let $I_j:=(u_1,\ldots,u_j)$, for $1\leq j\leq m$. 

Let $Z_1=\{x_1,\ldots,x_n\}$ and 
$Z_j=\{x_i\;:\;x_i\notin (I_{j-1} : u_j)\}$ for $2\leq j\leq m$.

Note that, for any $2\leq j\leq m$, we have
$$I_j/I_{j-1} = u_j (S/(I_{j-1}:u_j)) = u_j K[Z_j].$$
Hence the ideal $I$ has the Stanley decomposition
\begin{equation}\label{sdep}
I = u_1K[Z_1]\oplus u_2K[Z_2] \oplus \cdots \oplus u_mK[Z_m].
\end{equation}
According to \cite[Corollary 2.7]{varbaro}, the projective dimension of $S/I$ is
$$\pd(S/I)=\max\{n_j\;:\;2\leq j\leq m \}+1.$$
Hence, Ausl\"ander-Buchsbaum formula implies that
$$ \depth(S/I)=n-\max\{n_j\;:\;2\leq j\leq m \}-1= $$
\begin{equation}\label{dep}
= \min\{n-n_j\;:\;2\leq j\leq m\}-1=\min\{|Z_j|\;:\;2\leq j\leq m\}-1. 
\end{equation}
Note that, \eqref{sdep} and \eqref{dep} implies $\sdepth I\geq \depth I$, a fact 
which was proved in \cite{jah}.
We recall the following results:

\begin{prop}\label{depcreste}
 Let $I\subset S$ be a monomial ideal and $u\in S\setminus I$ a monomial. Then:
\begin{enumerate}
 \item[(1)] $\depth(S/(I:u))\geq \depth(S/I)$. (\cite[Corollary 1.3]{asia})
 \item[(2)] $\sdepth(S/(I:u))\geq \sdepth(S/I)$. (\cite[Proposition 2.7(2)]{sir})
\end{enumerate}
\end{prop}

\begin{prop}\label{p12}
Let $0\to U \to M \to N \to 0$ be a short exact sequence of finitely generated $\mathbb Z^n$-graded $S$-modules. Then 
$\sdepth(M)\geq \min\{\sdepth(U),\sdepth(N)\}$. (\cite[Lemma 2.2]{asia})
\end{prop}

Note that, a proper monomial ideal $I\subset S$ is principal if and only if $\depth(S/I)=n-1$ if and only if $\sdepth(S/I)=n-1$.

\begin{lema}\label{p13}
Let $I\subset S$ be a monomial ideal with linear quotients with $\depth(S/I)=n-2$. Then there exists some $i\in [n]$ and
a monomial $u\in G(I)$ such that $(I,x_i)=(u,x_i)$ and $(I:x_i)$ has linear quotients.
\end{lema}

\begin{proof}
If $n=2$ and $I=S$ then $u=1\in G(I)$ and the assertion is obvious. Hence, we may assume that $I$ is proper.

First, note that $I$ is not principal. Since $I$ has linear quotients, we can assume that
$G(I)=\{u_1,\ldots,u_m\}$ such that $((u_1, \ldots , u_{j-1}) : u_j)$ is generated by variables, 
    for every $2\leq j\leq m$. We consider the decomposition \eqref{sdep}, that is
    \begin{equation*}
     I= u_1K[Z_1]\oplus u_2K[Z_2] \oplus \cdots \oplus u_mK[Z_m],
    \end{equation*}
    where $Z_1=\{x_1,\ldots,x_n\}$ and $Z_j$ is the set of variables which do not belong to $((u_1, \ldots , u_{j-1}) : u_j)$,
    for $2\leq j\leq m$. From \eqref{dep}, it follows that $|Z_j|=n-1$ for $2\leq j\leq m$.
		Since $((u_1, \ldots , u_{j-1}) : u_j)=(x_i\;:\;x_i\notin Z_j)$, it follows that for any $2\leq j\leq m$ we have 
    \begin{equation}\label{desk}
     (u_1,\ldots,u_{j-1})\cap u_jK[Z_j] = \{0\}.
    \end{equation}
				
	  We assume, by contradiction, that for any $i \in [n]$ there exists $k_{i} > \ell_{i}\in [m]$ 
		such that $x_i\nmid u_{\ell_{i}}$. We claim that $x_i \in Z_{k_i}$. 
		Indeed, otherwise $Z_{k_{i}}=\{x_1,\ldots,x_n\}\setminus\{x_i\}$ and therefore 
		$u_{k_{i}}u_{\ell_{i}}\in  u_{\ell_{i}} S \cap u_{k_{i}} K[Z_{k_{i}}]$,
    a contradiction to \eqref{desk} for $j=k_{i}$. 
		
		Without any loss of generality, we can assume $k_n=\max\{k_i\;:\;i\in [n]\}$. Since $x_n\in Z_{k_n}$, it follows that
		$Z_{k_n}=\{x_1,\ldots,x_n\}\setminus\{x_t\}$ for some $t\leq n-1$. Since $x_t\in Z_{k_t}$ it follows that $k_t < k_n$
		and, moreover, $x_t\nmid u_{k_t}$, that is $u_{k_t}\in K[Z_{k_n}]$. Therefore
		$u_{k_{t}}u_{k_{n}}\in  u_{k_{t}} S \cap u_{k_{n}} K[Z_{k_{n}}]$,
    a contradiction to \eqref{desk} for $j=k_{n}$.		
		
		Thus, that there exists $i\in [n]$ such that for any $k_i>\ell_i\in[m]$, $x_i\mid u_{\ell_i}$. It implies that $x_i\mid u_j$ for $j=1,\ldots,m-1$.
		It follows that $(I:x_i)=(u'_1,\ldots,u'_m)$ where $u'_j=u_j/x_i$ for $j=1,\ldots,m-1$ and $u'_m=u_m$ if $x_i\nmid u_m$ and $u'_m=u_m/x_i$ if $x_i\mid u_m$. It is clear that
		$\{u'_1,\ldots,u'_{m-1}\}\subset G(I:x_i)$ and 		
		\begin{equation}\label{uuu}
		((u'_1,\ldots,u'_{j-1}):u'_j)=((u_1,\ldots,u_{j-1}):u_j)\text{ for all }2\leq j\leq m-1.
		\end{equation}
                We have $u'_m \in (u_m:x_i) \subseteq (I:x_i)$. If $u'_m\notin G(I:x_i)$ then $G(I:x_i)=\{u'_1,\ldots,u'_{m-1}\}$ and, from \eqref{uuu}, it follows that $(I:x_i)$ has linear quotients.                		
		On the other hand, assume $u'_m\in G(I:x_i)$ then we claim that $x_i\mid u_m$, hence $((u'_1,\ldots,u'_{m-1}):u'_m)=((u_1,\ldots,u_{m-1}):u_m)$ and, again, from \eqref{uuu}, 
		it follows that $(I:x_i)$ has linear quotients. Indeed, otherwise, $u'_m=u_m \in G(I:x_i)$ and $G(I:x_i)=\{u'_1,\ldots,u'_{m-1},u_m\}$, then for $1\leq j\leq m-1$ there exists 
                $\ell_j \in [m]\setminus [i]$ such that $x_{\ell_j}\mid u'_j$ and $x_{\ell_j}\nmid u_m$, and thus $x_{\ell_j},x_i \in \supp(u_i)\setminus\supp(u_m)$, which contradicts that
                $((u_1,\ldots,u_{m-1}):u_m)$ is generated by variables.
\end{proof}

\begin{teor}\label{teo1}
Let $I\subset S$ be a monomial ideal with linear quotients. If $\depth(S/I)=n-2$, then $\sdepth(S/I)=n-2$.
\end{teor}

\begin{proof}
    If $n=2$ then $S/I=0$ and there is nothing to prove, so we may assume $n\geq 3$ and $I$ is proper with 
		$G(I)=\{u_1,\ldots,u_m\}$ for some $m\geq 2$. We use induction on $m$ and
    $d=\sum_{j=1}^m \deg(u_i)$. If $m=2$, then from \cite[Proposition 1.6]{smal} it follows that $\sdepth(S/I)=n-2$.
    If $d=2$, then $I$ is generated by two variables and there is nothing to prove.

    Assume $m>2$ and $d>2$. According to Lemma \ref{p13}, there exist $i\in [n]$ such that $(I,x_i)=(u_m,x_i)$.
    Since $(I,x_i)=(u_m,x_i)$, from \cite[Proposition 1.2]{smal} 
    it follows that $\sdepth(S/(I,x_i))\geq n-2$. If $(I:x_i)$ is principal, then $\sdepth(S/(I:x_i))=\depth(S/(I:x_i))=n-1$.
		
		Assume that $(I:x_i)$ is not principal. We have that $\depth(S/(I:x_i))\leq n-2$. On the other hand, by 
		Proposition \ref{depcreste}(1) we have $\depth(S/(I:x_i))\geq \depth(S/I)=n-2$ and thus $\depth(S/(I:x_i))=n-2$.
		From the proof of Lemma \ref{p13}, we have $G(I:x_i)\subset \{u_1/x_i,\ldots,u_{m-1}/x_i,u_m\}$.
    It follows that $$d':=\sum_{u\in G(I:x_i)}\deg(u) < d,$$ 
    thus, by induction hypothesis, we have $\sdepth(S/(I:x_i))=n-2$. In both cases, $$\sdepth(S/(I:x_i))\geq n-2.$$
    From Proposition \ref{p12} and the short exact sequence $$0 \to S/(I:x_i) \to S/I \to S/(I,x_i) \to 0,$$
    it follows that $\sdepth(S/I)\geq \min\{ \sdepth(S/(I:x_i)),\sdepth(S/(I,x_i))\} \geq n-2$. Since $I$ is not principal, 
		it follows that $\sdepth(S/I)=n-2$, as required.
\end{proof}

\begin{lema}\label{lemas}
Let $I\subset S$ be a monomial ideal and $u\in S$ a monomial with $(I:u)=(x_1,\ldots,x_m)$.
Assume that $S/I$ has a Stanley decomposition \begin{equation}\label{desco}
\mathcal D:S/I=\bigoplus_{i=1}^r v_iK[Z_i],
\end{equation}
such that there exists $i_0$ with $Z_{i_0}=\{x_{m+1},\ldots,x_n\}$ and $v_{i_0}\mid u$. Then:
$$\sdepth(S/(I,u))\geq \min\{\sdepth(\mathcal D),n-m-1\}.$$
\end{lema}

\begin{proof}
If $\sdepth(S/I)=0$ or $m=n-1$, then there is nothing to prove.
We assume that $\sdepth(S/I)\geq 1$ and $m\leq n-2$. Since $S/(I:u) = S/(x_1,\ldots,x_m) \cong K[x_{m+1},\ldots,x_n]$, from the short exact sequence 
$$ 0 \to S/(I:u) \stackrel{\cdot u}{\longrightarrow} S/I \longrightarrow S/(I,u) \longrightarrow 0,$$
it follows that we have the $K$-vector spaces isomorphism
\begin{equation}\label{labe}
S/I \cong S/(I,u)\oplus uK[x_{m+1},\ldots,x_n]. 
\end{equation}
From our assumption, $uK[x_{m+1},\ldots,x_n]=uK[Z_{i_0}]\subset v_{i_0}K[Z_{i_0}]$. Hence, from \eqref{desco} and \eqref{labe} it follows that
\begin{equation}\label{coor}
 S/(I,u) \cong \left(\bigoplus_{i\neq i_0}v_i K[Z_i]\right) \oplus \frac{v_{i_0}K[Z_{i_0}]}{uK[Z_{i_0}]} 
\cong \left(\bigoplus_{i\neq i_0}v_i K[Z_i]\right) \oplus \frac{K[x_{m+1},\ldots,x_n]}{w_{0}K[x_{m+1},\ldots,x_n]},
\end{equation}
where $w_0=\frac{u}{v_{i_0}}$. On the other hand, $\sdepth\left(\frac{K[x_{m+1},\ldots,x_n]}{w_{0}K[x_{m+1},\ldots,x_n]}\right)=n-m-1$. 
Hence \eqref{coor} and Proposition \ref{p12} yields the required conclusion.
\end{proof}

\begin{teor}\label{teo2}
Let $I\subset S$ be a monomial ideal with linear quotients, $G(I)=\{u_1,\ldots,u_m\}$. Let $I_j=(u_1,\ldots,u_j)$ for $1\leq j\leq m$, such that
$(I_{j-1}:u_j)=(\{x_1,\ldots,x_n\}\setminus Z_j)$, where $Z_j\subset\{x_1,\ldots,x_n\}$, for all $2\leq j\leq m$. 

We assume that for any $2\leq j\leq m$,
there exists a Stanley decomposition $\mathcal D_{j-1}$ of $S/I_{j-1}$ such that $\sdepth(\mathcal D_{j-1}) = \sdepth(S/I_{j-1})$ and there exists a
Stanley subspace $w_{j-1}K[W_{j-1}]$ of $\mathcal D_{j-1}$ with $w_{j-1}\mid u_j$ and $W_{j-1}=Z_j$. 

Then $\sdepth(S/I)\geq \depth(S/I)$.
\end{teor}

\begin{proof}
From the hypothesis and Lemma \ref{lemas}, we have that
$$ \sdepth(S/I_j) = \sdepth(S/(I_{j-1},u_j)) \geq \min\{ \sdepth(\mathcal D_{j-1}), n-n_i-1 \} $$
\begin{equation}\label{cucx}
 = \min\{ \sdepth(S/I_{j-1}), n-n_j-1\} \},\text{ for all }2\leq j\leq m, 
\end{equation}
where $n_j=n-|Z_j|$, $1\leq j\leq m$. On the other hand, according to \eqref{dep}, 
\begin{equation}\label{cucus}
\depth(S/I)=\min_{j=2}^m\{n-n_j-1\}.
\end{equation}
Since $\sdepth(S/I_1)=\depth(S/I_1)=n-1$, by applying repeatedly \eqref{cucx} we deduce that
$$\sdepth(S/I)=\sdepth(S/I_m)\geq \min_{j=2}^m\{n-n_j-1\}.$$
Hence, from \eqref{cucus} we get the required conclusion.
\end{proof}

\begin{exm}\rm
Let $I=(x_1^2,x_1x_2^2,x_1x_2x_3^2)\subset S=K[x_1,x_2,x_3,x_4]$. Let $u_1=x_1^2$, $u_2=x_1x_2^2$ and $u_3=x_1x_2x_3^2$. Since $((u_1):u_2)=(x_1)$ and
$((u_1,u_2):u_3)=(x_1,x_2)$, it follows that $I$ has linear quotients with repect to the order $u_1\leqslant u_2\leqslant u_3$. 
Moreover, $$I=u_1K[Z_1]\oplus u_2K[Z_2] \oplus u_3K[Z_3] = x_1^2K[x_1,x_2,x_3,x_4]\oplus x_1x_2^2K[x_2,x_3,x_4]\oplus x_1x_2x_3^2K[x_3,x_4].$$
Let $I_1=(u_1)$ and $I_2=(u_1,u_2)$. We consider the Stanley decomposition $$\mathcal D_1:\; S/I_1=K[x_2,x_3,x_4]\oplus x_1K[x_2,x_3,x_4],$$ 
of $S/I_1$ with $\sdepth(\mathcal D_1)=\sdepth(S/I_1)=3$. 
Let $w_1=x_1$ and $W_1=\{x_2,x_3,x_4\}$. Clearly, $W_1=Z_2$ and $w_1\mid u_2$.
As in the proof of Lemma \ref{lemas}, we obtain the Stanley decomposition 
$$\mathcal D_2:\; S/I_2 = K[x_2,x_3,x_4] \oplus \frac{x_1 K[x_2,x_3,x_4]}{x_1x_2^2K[x_2,x_3,x_4]} = K[x_2,x_3,x_4] \oplus x_1K[x_3,x_4] \oplus x_1x_2K[x_3,x_4]$$
of $S/I_2$ with $\sdepth(\mathcal D_2)=\sdepth(S/I_2)=2$. 

Let $w_2=x_1x_2$ and $W_2=\{x_3,x_4\}$. Clearly, $W_2=Z_3$ and $w_2\mid u_3$. Hence, according to Theorem \ref{teo2}, $\sdepth(S/I)\geq \depth(S/I)=1$. Note that $$\mathcal D\;:\;S/I = K[x_2,x_3,x_4] \oplus x_1K[x_3,x_4] \oplus x_1x_2K[x_4] \oplus x_1x_2x_3K[x_4],$$
 is a Stanley decomposition of $S/I$ with $\sdepth(\mathcal D)=1$ and thus $\sdepth(S/I)\geq 1$.

On the other hand, since $(x_1,x_2,x_3)$ is an associated prime to $S/I$, it follows that $\sdepth(S/I)\leq 1$ and thus
$\sdepth(S/I)=1$. Finally, note that 
$$\depth(S/I_2)=2,\; (I_2,x_1)=(x_1)\text{ and }(I_2:x_1)=(x_1,x_2^2).$$
In particular, we have $\sdepth(S/(I_2:x_1))=\depth(S/(I_2:x_1))=2$, while $\sdepth(S/(I_2,x_1))=\depth(S/(I_2,x_1))=3$.
\end{exm}

We propose the following conjecture:

\begin{conj}\label{conj}
If $I\subset S$ is a proper monomial ideal with linear quotients, then there exists $i\in[n]$ such that
$\depth(S/(I,x_i)) \geq \depth(S/I)$.
\end{conj}

The following result is well know in literature. However, in order of completeness, we give a proof.

\begin{lema}\label{lemaa2}
Let $I\subset S$ be a monomial ideal with linear quotients and $x_i$ a variable. Then $(x_i,I)$ has linear quotients.
Moreover, if $S'=K[x_1,\ldots,x_{i-1},x_{i+1},\ldots,x_n]$, then $(x_i,I)=(x_i,J)$, where $J\subset S'$ is a monomial ideal with linear quotients.
\end{lema}

\begin{proof}
We consider the order  $u_1 \leqslant u_2 \leqslant \cdots  \leqslant u_m$  on $G(I)$, such that, for every $2 \leq j \leq m$, the ideal $(I_{j-1} : u_j)$ is generated by a 
nonempty subset $\bar Z_j$ of variables. We assume that $u_{j_1}\leqslant u_{j_2} \leqslant \cdots \leqslant u_{j_p}$ are the minimal monomial generators of $I$ which are not multiple of $x_i$.
We have that $((x_i):u_{j_1})=(x_i)$. Also, for $2\leq k \leq p$, we claim that 
\begin{equation}\label{claim}
 ((x_i,u_{j_1},\ldots,u_{j_{k-1}}):u_{j_k}) = (x_i,\bar Z_{j_k}).
\end{equation}
Indeed, since $((u_1,\ldots,u_{j_k-1}):u_{j_k})=(\bar Z_{j_k})$ and $x_iu_{j_k}\in (x_i,u_{j_1},\ldots,u_{j_{k-1}})$ it follows that 
$(x_i,\bar Z_{j_k})\subset ((x_i,u_{j_1},\ldots,u_{j_{k-1}}):u_{j_k})$.
Conversely, assume that $v\in S$ is a monomial with $vu_{j_k}\in (x_i,u_{j_1},\ldots,u_{j_{k-1}})=(x_i,u_1,\ldots,u_{j_k-1})$. If $x_i\nmid v$, then $vu_{j_k}\in(u_1,\ldots,u_{j_k-1})$, hence
$v\in (\bar Z_{j_k})$. If $x_i\mid v$, then $v\in (x_i, \bar Z_{j_k})$.
Hence the claim \eqref{claim} is true and therefore $(x_i,I)$ has linear quotients.
Now, let $J=(u_{j_1},\ldots,u_{j_p})$. For any $2\leq k\leq p$, we have that 
\begin{equation}\label{zece}
((u_{j_1},\ldots,u_{j_{k-1}}):u_{j_k})\subset ((u_1,\ldots,u_{j_k-1}):u_{j_k})=(\bar Z_{j_k}). 
\end{equation}
From \eqref{claim} and \eqref{zece}, one can easily deduce that $((u_{j_1},\ldots,u_{j_{k-1}}):u_{j_k})=(\bar Z_{j_k}\setminus \{x_i\})$.
Hence, $J$ has linear quotients.
\end{proof}

\begin{obs}\rm
Let $I\subset S$ be a monomial ideal with linear quotients, 
$G(I)=\{u_1,\ldots,u_m\}$, $I_j=(u_1,\ldots,u_j)$ for $1\leq j\leq m$, such that $(I_{j-1}:u_j)=(\{x_1,\ldots,x_n\}\setminus Z_j)$, 
where $Z_j\subset\{x_1,\ldots,x_n\}$, for all $2\leq j\leq m$. $I$ has the Stanley decomposition:
$$I=u_1K[Z_1]\oplus u_2K[Z_2]\oplus \cdots \oplus u_mK[Z_m],$$
where $Z_1=\{x_1,\ldots,x_n\}$. We have that 
$$\depth(S/I)=n-s-1,\text{ where }n-s= \min\{|Z_j|\;:\;1\leq j\leq m\}.$$
We claim that Conjecture \ref{conj} is equivalent to the fact that there exists $i\in [n]$ such that there is no $1\leq j\leq m$ with 
$x_i\nmid u_j$, $x_i\in Z_j$ and $|Z_j|=n-s$. Indeed, with the notations of Lemma \ref{lemaa2}, if there is some $u_{j_k}$ with
$x_i\nmid u_{j_k}$ and $x_i\in Z_{j_k}$ then $u_{j_k}K[Z_{j_k}\setminus \{x_i\}]$ is a subspace in the
decomposition of the ideal with linear quotients $J \subset S'=K[x_1,\ldots,x_{i-1},x_{i+1},\ldots,x_n]$ and thus 
$$\depth(S/(I,x_i))=\depth(S'/J)\leq (n-1)-s-1=n-s-2<\depth(S/I).$$ The converse is similar.
\end{obs}

We propose a stronger form of Conjecture \ref{conj}.

\begin{conj}\label{conjj}
If $I\subset S$ is a proper monomial ideal with linear quotients, then there exists $i\in[n]$ such that:
\begin{enumerate}
\item[i)] $\depth(S/(I,x_i)) \geq \depth(S/I)$ and 
\item[ii)] $\sdepth(S/(I,x_i))\leq \sdepth(S/I)$.
\end{enumerate}
\end{conj}

Note that, if $x_i$ is a minimal generator of $I$, then conditions i) and ii) from Conjecture \ref{conjj} are trivial.


\begin{teor}\label{teo3}
If Conjecture \ref{conjj} is true and $I\subset S$ is a proper monomial ideal with linear quotients, then
$\sdepth(S/I)\geq \depth(S/I)$.
\end{teor}

\begin{proof}
We use induction on $n\geq 1$. If $n=1$ then there is nothing to prove. Assume $n\geq 2$.
Let $I\subset S$ be a monomial ideal with linear quotients and let $i\in[n]$ such that $\depth(S/(I,x_i)) \geq \depth(S/I)$ 
and $\sdepth(S/(x_i,I))\leq \sdepth(S/I)$. 
We consider the short exact sequence 
\begin{equation}\label{mast}
 0 \to \frac{S}{(I:x_i)} \to \frac{S}{I} \to \frac{S}{(I,x_i)} \to 0.
\end{equation}
Let $S':=K[x_1,\ldots,x_{i-1},x_{i+1},\ldots,x_n]$. According to Lemma \ref{lemaa2}, $(x_i,I)=(x_i,J)$ where $J\subset S'$ is a monomial 
ideal with linear quotients. Note that: 
$$\sdepth(S/(x_i,I))=\sdepth(S/(x_i,J))=\sdepth(S'/J)\text{ and }\depth(S/(I,x_i))=\depth(S'/J).$$ 
From the induction hypothesis, we have $\sdepth(S'/J)\geq \depth(S'/J)$. It follows that:
\begin{align*}
& \sdepth(S/I) \geq \sdepth(S/(I,x_i))=\sdepth(S'/J) \\ 
& \geq \depth(S'/J)=\depth(S/(I,x_i))\geq \depth(S/I),
\end{align*}
as required.
\end{proof}

\begin{obs}\rm
Note that, if $I\subset S$ has linear quotients, then $(I:x_i)$ has not necessarily the same property. 
For example, the ideal $I=(x_1x_2,x_2x_3x_4,x_3x_4x_5)\subset K[x_1,\ldots,x_5]$ has linear quotients, but
$(I:x_5)=(x_1x_2,x_3x_4)$ has not. Henceforth, in the proof of Theorem \ref{teo2}, we cannot argue, inductively, 
that $\sdepth(S/(I:x_i))\geq \depth(S/(I:x_i))$.
\end{obs}

\section{Remarks on the Hilbert depth}\label{s3}

Let $I=(u_1,\ldots,u_m)\subset S$ be a proper squarefree monomial with linear quotients, where $(u_1,\ldots,u_{i-1}):u_i$ is 
generated by variables for any $2\leq i\leq m$.
As we seen in the previous section, $I$ has a decomposition
\begin{equation}\label{ecuu}
I=u_1K[Z_1]\oplus u_2K[Z_2] \oplus \cdots \oplus u_mK[Z_m].
\end{equation}
Moreover, since $I$ is squarefree, $Z_1=\{x_1,\ldots,x_n\}$ and, for $2\leq i\leq m$, $Z_i$ consists in the variables which are not in 
$(u_1,\ldots,u_{i-1}):u_i$, it follows that $\supp(u_i)\subset Z_i$ for all $1\leq i\leq m$. 
Therefore, if we denote $d_i=\deg(u_i)$
and $n_i=|Z_i|$, then $d_i\leq n_i$, for all $1\leq i\leq m$.

We use the convention $\binom{r}{s}=0$ for $s<0$.

\begin{lema}\label{l21}
With the above notations, we have that:
\begin{enumerate}
\item[(1)] $\alpha_j(I)=\sum\limits_{i=1}^m \binom{n_i-d_i}{j-d_i}$ for all $0\leq j\leq n$.
\item[(2)] $\alpha_j(S/I)=\binom{n}{j}-\sum\limits_{i=1}^m \binom{n_i-d_i}{j-d_i}$ for all $0\leq j\leq n$.
\end{enumerate}
\end{lema}

\begin{proof}
(1) For convenience, we assume that $u_1=x_1x_2\cdots x_p$ for some $p\leq n$. For $j\geq p$, a squarefree monomial of degree $j$
in $u_1K[Z_1]=u_1K[x_1,\ldots,x_n]$ is of the form $v=u_1w$, where $w\in K[x_{p+1},\cdots,x_n]$ is squarefree of degree $j-p$. Hence,
there are $\binom{n-p}{j-p}=\binom{n_1-d_1}{j-d_1}$ such monomials. Similarly, there are $\binom{n_i-d_i}{j-d_i}$ squarefree
monomials of degree $j$ in $u_iK[Z_i]$ for all $2\leq i\leq m$. Hence, we get the required conclusion from \eqref{ecuu}.

(2) It follows immediately from (1).
\end{proof}

We recall the following combinatorial identity, which can be easily derived from the Chu-Vandermonde identity
\begin{equation}\label{magic}
\sum_{j=0}^k (-1)^{k-j} \binom{d-j}{k-j}\binom{n}{j} = \binom{n-d+k-1}{k}.
\end{equation}
Now, we state the following result, which follows immediately from Lemma \ref{l21} and \eqref{magic}:

\begin{prop}\label{p22}
With the above notations, we have that:
\begin{enumerate}
\item[(1)] $\beta_k^d(I)=\sum\limits_{i=1}^m \sum\limits_{j=0}^k (-1)^{k-j} \binom{d-j}{k-j}\binom{n_i-d_i}{j-d_i}$ 
           for all $0\leq k\leq d\leq n$.
\item[(2)] $\beta_k^d(S/I)=\binom{n-d+k-1}{k}-\sum\limits_{i=1}^m \sum\limits_{j=0}^k (-1)^{k-j} 
\binom{d-j}{k-j}\binom{n_i-d_i}{j-d_i}$ for all $0\leq k\leq d\leq n$.
\end{enumerate}
\end{prop}

If $k\geq D$ then, using \eqref{magic} and taking $\ell=j-D$ we get 
\begin{equation}\label{hyper}
\sum\limits_{j=0}^k (-1)^{k-j} \binom{d-j}{k-j}\binom{N-D}{j-D}= \sum_{\ell=0}^{k-D}(-1)^{k-D-\ell} \binom{d-D-\ell}{k-D-\ell}
\binom{N-D}{\ell} = \binom{N-d+k-D-1}{k-D}.
\end{equation}
Note that \eqref{hyper} is trivially satisfied for $k<D$ also.

From Proposition \ref{p22} and \eqref{hyper} we get the following:

\begin{cor}\label{cor23}
With the above notations, we have that:
\begin{enumerate}
\item[(1)] $\beta_k^d(I)= \sum\limits_{i=1}^m \binom{n_i-d+k-d_i-1}{k-d_i}$ for all $0\leq k\leq d\leq n$.
\item[(2)] $\beta_k^d(S/I)=\binom{n-d+k-1}{k}- \sum\limits_{i=1}^m \binom{n_i-d+k-d_i-1}{k-d_i}$ for all $0\leq k\leq d\leq n$.
\end{enumerate}
\end{cor}

The problem of computing $\qdepth(I)$ and $\qdepth(S/I)$ using directly the formulas given in Corollary \ref{cor23} seems
hopeless. However, we can tackle the following particular case:

\begin{teor}\label{teo4}
Let $I\subset S$ be a proper squarefree monomial ideal with linear quotients with $\depth(S/I)=n-2$. Then 
$$\qdepth(S/I)=\sdepth(S/I)=n-2.$$
\end{teor}

\begin{proof}
From Theorem \ref{teo1} and \eqref{qdep} it follows that 
$$\qdepth(S/I)\geq \sdepth(S/I)=n-2.$$
Hence, in order to complete the proof it is enough to show that $\qdepth(S/I)\leq n-2$.
If $\alpha_{n-1}(S/I)=0$ then, according to \eqref{alfakk}, there is nothing to prove.

Suppose that $\alpha_{n-1}(S/I)=s>0$. 
From \cite[Lemma 2.1]{jahz} we can assume
that $\deg(u_1)\leq \deg(u_2)\leq \cdots \leq \deg(u_m)$, where   $u_1 \leqslant u_2 \leqslant \cdots  \leqslant u_m$ 
is the linear order on $G(I)$. If $m=1$ then $I=(u_1)$ is principal, a contradiction with the hypothesis $\depth(S/I)=n-2$.

Note that, if $x_1x_2\cdots x_n\in G(I)$ then, since $I$ has linear quotients, it follows that $u_1=x_1x_2\cdots x_n$ and $I=(u_1)$, 
a contradiction. Therefore
$$\deg(u_1)\leq \deg(u_2)\leq \cdots \leq \deg(u_m)\leq n-1.$$
We claim that $\deg(u_1)\geq s$. Assume by contradiction that $\deg(u_1)=\ell<s$ and let's say that $u_1=x_1\cdots x_{\ell}$.
Then $v_k=x_1\cdots x_n/x_k \in I$ for all $\ell < k\leq n$ and thus $\alpha_{n-1}(S/I)\leq \ell$, a contradiction.
In particular, we have $\alpha_j(S/I)=\binom{n}{j}$ for all $j\leq s-1$ and thus, from 
\eqref{betak} and \eqref{magic}, it follows that
\begin{equation}\label{equ1}
\beta_k^{n-1}(S/I)=\sum_{j=0}^k (-1)^{k-j}\binom{n-1-j}{k-j}\binom{n}{j}=1\text{ for all }k\leq s-1.
\end{equation}
We assume by contradiction that $\hdepth(S/I)=n-1$. From \eqref{alfak}, it follows that 
\begin{equation}\label{equ2}
s=\alpha_{n-1}(S/I) = \sum_{j=0}^{n-1}\beta_j^{n-1}(S/I)\text{ with }\beta_j^{n-1}(S/I)\geq 0.
\end{equation}
Therefore, from \eqref{equ1} we get 
\begin{equation}\label{equ3}
\beta_j^{n-1}(S/I)=0\text{ for all }s\leq j\leq n-1.
\end{equation}
From \eqref{alfak}, \eqref{equ1} and \eqref{equ3} it follows that 
\begin{equation}\label{alfa}
\alpha_k(S/I)=\sum_{j=0}^k \beta_j^{n-1}(S/I)\binom{n-1-j}{k-j} = \binom{n}{k}-\binom{n-s}{k-s}\text{ for all }0\leq k\leq n.
\end{equation}
From Lemma \ref{l21}(2) and \eqref{alfa} it follows 
\begin{equation}\label{s1}
\sum\limits_{i=1}^m \binom{n_i-d_i}{s-d_i} = 1.
\end{equation}
Since $s\leq d_1\leq d_2\leq \cdots \leq d_m$ and $d_i\leq n_i$ for all $1\leq i\leq m$, from \eqref{s1} it follows that $d_1=s$ and $d_i>s$ for $2\leq i\leq m$.
Since $n_1=n$, from Lemma \ref{l21}(2) it follows that
$$\alpha_{d_2}(S/I)=\binom{n}{d_2}-\sum\limits_{i=1}^m \binom{n_i-d_i}{d_2-d_i}\leq \binom{n}{d_2}-\binom{n-d_1}{d_2-d_1}-\binom{n_2-d_2}{0} =
\binom{n}{d_2}-\binom{n-s}{d_2-s}-1,$$
which contradicts \eqref{alfa}.
\end{proof}


\subsection*{Acknowledgements}

We would like to express our gratitude to the anonymous referee who help us to correct, improve and clarify our manuscript.

The second author was supported by a grant of the Ministry of Research, Innovation and Digitization, CNCS - UEFISCDI, 
project number PN-III-P1-1.1-TE-2021-1633, within PNCDI III.





{}

\end{document}